\nonstopmode
\documentclass[10pt,reqno]{amsart}
\usepackage{latexsym}
\usepackage{fancyhdr}
\usepackage{amsmath, amssymb,amsthm}
\usepackage[all]{xy}
\usepackage{pdflscape}
\usepackage{longtable}
\usepackage{rotating}
\usepackage{verbatim}
\usepackage{hyperref}
\usepackage{cleveref}
\usepackage{subfigure}
\usepackage{mathrsfs}
\usepackage{mdwlist}
\usepackage{dsfont}
\usepackage{bbold}
\usepackage{mathtools}
\usepackage{booktabs}
\usepackage{float}
\usepackage{color}

\theoremstyle{plain}
\newtheorem{teor}{Theorem}[]

\newtheorem{proposition}[teor]{Proposition}

\theoremstyle{definition}

\newtheorem*{remark*}{Remark}

\begin{document}

%\title[Bernoulli and Euler numbers through divergent series]{A note on some recurrence formulas for Bernoulli and Euler numbers through divergent series}

\title[Bernoulli and Euler numbers from divergent series]{Bernoulli and Euler numbers from divergent series}

\author{Sergio A. Carrillo}
\address{(Sergio A. Carrillo) Fakult\"at f\"ur Mathematik, Universit\"at Wien, Oskar-Morgenstern-Platz~1, A-1090 Wien, Austria. Escuela de Ciencias Exactas e Ingenier\'{i}a, Universidad Sergio Arboleda, Calle 74 14-14, Bogot\'{a}, Colombia.}

\email{sergio.carrillo@univie.ac.at, sergio.carrillo@usa.edu.co}

\thanks{The author was supported by the Austrian FWF-Project P 26735-N25.}


%\begin{abstract}The aim of this note is to provide a simple proof of some well-known identities and recurrences relating classical Bernoulli and Euler numbers by using the Abel sum of the divergent series $\sum_{n=0}^\infty (-1)^{n} (n+1)^k$,  $k$ a positive integers. Special attention is placed on the fact that the numerical value of these sums is determined by the linearity of the summation method involved.\end{abstract}

%\keywords{Bernoulli numbers, Euler numbers, Abel summability}

\maketitle

Many special numeric sequences have been studied intensively due to their appearance and applications in Combinatorics, Number Theory, and Analysis. For instance, Fibonacci, Bernoulli, Euler, Eulerian \cite{Petersen}, or Stirling numbers \cite{Stanley}. In addition to common techniques to obtain properties and relations among them, summation of divergent series can be used too!

The Bernoulli and Euler numbers are two sequences of rational numbers that play an important role in mathematical analysis. They appear naturally as the coefficients of the Taylor expansions of trigonometric functions and in the computation of sums of series and asymptotic expansions. For instance, in the calculation of $\zeta(2k)$, where $k$ is an integer and $\zeta$ denotes the Riemann zeta function \cite{Osler}, or in the Euler-Maclaurin summation formula \cite{Hardy} Chapter XIII. They also exhibit interesting relations between other numbers, for example with Euler's constant \cite{Lih}. Bernoulli numbers have been called an ``unifying force'' in mathematics due to their presence in several branches as in Analytic Number Theory \cite{BN} and Differential Topology \cite{Mazur}.

The \textit{Bernoulli numbers} $B_n$ (with signs) were introduced by J. Bernoulli in his \textit{Ars Conjectandi} \cite{Bernoulli,Sylla}, published posthumously in 1713, to give a precise formula for $1^k+2^k+\cdots+n^k$ as a polynomial in $n$. They are defined as the coefficients in the Taylor expansion at $z=0$ of \begin{equation}\label{z/(ez-1)+z/2}\frac{z}{e^z-1}=-\frac{z}{2}+\frac{z}{2}\frac{e^z+1}{e^z-1}=\sum_{k=0}^\infty \frac{B_k}{k!}z^k,\quad |z|<2\pi.\end{equation} From the  definition it is deduced that $B_0=1$, $B_1=-\frac{1}{2}$, $B_2=\frac{1}{6}$ and $\sum_{k=0}^{n-1}{n\choose k} B_k=0,$ $n\geq 2$. This formula allows to find $B_n$ recursively and it shows that $B_n$ is always rational. Also, since $\frac{z}{2}\frac{e^z+1}{e^z-1}$ is an even function, we deduce that $B_{2k+1}=0$, for $k\geq1$.

The \textit{Euler numbers} are defined as the Taylor expansion at $z=0$ of \begin{equation}\label{E_n def}\sec(z)=\frac{1}{\cos(z)}=\sum_{n=0}^\infty \frac{E_{n}}{n!}z^{n}=1+\sum_{n=1}^\infty \frac{E_{2n}}{(2n)!}z^{2n},\quad |z|<\frac{\pi}{2}.\end{equation} Then it is clear that $\sum_{k=0}^n {2n\choose 2k} (-1)^{k} E_{2k}=0,$ and thus the  $E_{2k}$ are integers.

Replacing $z$ by $2iz$ in equation (\ref{z/(ez-1)+z/2}) we find the Taylor expansion at $z=0$ of \begin{equation*}\label{Taylor expantion of cot}z\cot(z)=\sum_{k=0}^\infty \frac{B_{2k}}{(2k)!}(-1)^k 2^{2k}z^{2k},\quad |z|<\pi,\end{equation*} and from the identity $\tan(z)=\cot(z)-2\cot(2z)$, we see that \begin{equation*}\label{Taylor tan}\tan(z)=\sum_{n=1}^\infty (-1)^{n+1}\frac{2^{2n}(2^{2n}-1)B_{2n}}{(2n)!}z^{2n-1}=-2i\sum_{k=1}^\infty \frac{2^{k+1}-1}{k+1}B_{k+1}\frac{(2iz)^k}{k!},\quad |z|<\frac{\pi}{2}.\end{equation*} Thus the Maclaurin series of $\tan(z)$ and $\sec(z)$ are not trivial but there are elementary recursive methods to obtain them \cite{Lawson}. Finally, we obtain from the last equation the power series expansion \begin{equation}\label{Last formula}  -\frac{1}{1+e^z}=\sum_{k=0}^\infty \frac{2^{k+1}-1}{k+1}B_{k+1}\frac{z^k}{k!},\quad |z|<\pi.\end{equation} 

There are many lists of recurrences satisfied by the Bernoulli and Euler numbers, e.g., Nielsen's classical book \cite{Nielsen}. We only need one, namely
\begin{equation}\label{Prop B_n for 1^k-2^k+...}
\frac{2^{k+1}-1}{k+1}B_{k+1}=\frac{1}{2}-\sum_{l=1}^{k}{k\choose l}\frac{2^{l+1}-1}{l+1}B_{l+1},\hspace{0.4cm} k\geq1.
\end{equation} It can be obtained from the equality  $\frac{-1}{1+e^z}e^z=-1+\frac{1}{1+e^z}$ and the formula (\ref{Last formula}) by equating the corresponding coefficients of $z^k$. 

We will show how to obtain (\ref{Prop B_n for 1^k-2^k+...}) also by summing divergent series. The use of this method is not new. Garabedian  \cite{Garabedian}, for instance,  showed that \begin{equation*}\label{Explicit Bn}B_{n+1}=\frac{(-1)^n(n+1)}{2^{n+1}-1}\sum_{k=1}^{n+1} \frac{1}{2^k}\sum_{j=0}^{k-1} {k-1\choose j}(-1)^j(j+1)^n,\end{equation*} from summing $$\sigma_k:=1^k-2^k+3^k-4^k+\cdots,\quad k\in\mathbb{N},$$ using Ces\`{a}ro and Abel summability. R\c{a}dkowsk \cite{Ruso} also provided a proof using calculus of finite differences. Similarly, Namias \cite{Namias} deduced some other recurrences using Stirling's asymptotic series and the duplication formula for the Gamma function, although the results can be also obtained in an elementary way \cite{Deeba}. We will use the sum of $\sigma_k$ and the linearity of a summation method that can sum it to obtain (\ref{Prop B_n for 1^k-2^k+...}) and some other simple recurrences for Bernoulli and Euler numbers.

Let us recall that a series $\sigma=\sum_{n=0}^\infty a_n$ is said to be \textit{Abel summable} with sum $A(\sigma)$ if for all $x\in\mathbb{R}$ with $0\leq x<1$, the \textit{associated generating series} $\sum_{n=0}^\infty a_n x^n$ is convergent and $A(\sigma):=\lim_{x\rightarrow 1^-} \sum_{n=0}^\infty a_n x^n$ exists. In particular, the series $\sigma_k$ is Abel summable \cite{Varadarajan,Knopp}: if we replace $x=e^{-y}$ in the series  $$1^kx-2^kx^2+3^kx^3-\cdots=\sum_{n=0}^\infty (-1)^n(n+1)^k e^{-(n+1)y}=(-1)^k \frac{d^k}{dy^k}\left(\frac{1}{1+e^y}\right),$$ we can use the expansion (\ref{Last formula})  and take $y\rightarrow0^+$ to find the value $A(\sigma_k)=\frac{2^{k+1}-1}{k+1}B_{k+1}$, $k\geq1$. It is also clear that $A(\sigma_0)=\frac{1}{2}$. In the same way \begin{equation}\label{Abel sum of 1k-3k+5k+...}A(1^k-3^k+5^k-7^k+\cdots)=(-1)^{\lfloor k/2\rfloor}E_{k}/2,
\end{equation} where $\lfloor ,\rfloor$ denotes the floor function. Indeed, setting $x=e^{-2y}$ in the generating series $\sum_{n=0}^\infty (-1)^n (2n+1)^kx^n=(-1)^ke^y\frac{d^k}{dy^k}\left(\frac{1}{e^y+e^{-y}}\right)$, we can let $y\rightarrow0^+$ and then the formula follows from equation (\ref{E_n def}).

Is it possible to attribute a sum to a divergent series in a way compatible with the usual rules of calculus? For Euler the answer was positive! This is evidenced in his work  \textit{De seriebus divergentibus} \cite{Euler} on the \textit{Wallis series} $\sum_{n=0}^{\infty}(-1)^nn!$, where he found the sum $\int_0^\infty \frac{e^{-t}}{1+t}dt\approx0.5963473625$ \cite{Varadarajan}. In the same spirit, this was the belief of Hardy as he exhibited in his book \textit{Divergent Series} \cite{Hardy}. Nowadays, the theory of summability attempts to answer this question. 

We denote by $\mathcal{D}$ the $\mathbb{C}-$vector space of complex sequences $(a_n)_{n\geq 0}$ and by $\mathcal{C}$ subspace of sequences such that $\lim_{n\rightarrow+\infty} a_0+\cdots+a_n$ exists. We can think of elements of $\mathcal{D}$ as formal numerical series $\sigma=\sum_{n=0}^\infty a_n$. The space $\mathcal{C}$ is the domain of the \textit{sum homomorphism} $S:\mathcal{C}\rightarrow\mathbb{C}$, which associates a series $\sigma$ to its sum $S(\sigma)$, i.e., the limit of its partial sums. From this point of view, a summability method is a map $S^*:\mathcal{C}'\rightarrow\mathbb{C}$ on some linear subspace $\mathcal{C}\subseteq\mathcal{C}'\subseteq\mathcal{D}$ such that the following rules are satisfied:

\begin{enumerate}\item[1.] \textbf{Regularity rule}: If $\sigma\in\mathcal{C}$, then $S^*(\sigma)=S(\sigma)$.
	
	\item[2.] \textbf{Translation rule}: $S^*\left(\sum_{n=0}^\infty a_n\right)=a_0+S^*\left(\sum_{n=1}^\infty a_n\right)$.
	
	\item[3.] \textbf{Linearity rule}: $S^*$ is a $\mathbb{C}-$linear map.
	
\end{enumerate}

Ces\`{a}ro and Abel summability are examples of summability methods satisfying such rules. It is worth noting that such axioms were implicitly used by Euler.

We will use the following fact: \begin{center}\textit{Assume that $S^\ast$ satisfies the above rules. Then if it sums a series, the value we find through the rules is the value $S^\ast$ assigns to the series.}\end{center}  As a first example, we consider $F=1+1+2+3+5+\cdots=\sum_{n=0}^\infty F_n$, the series of Fibonacci numbers: if $S^*$ sums $F$, then $S^*({F})=2+S^*\left(\sum_{n=2}^\infty F_n\right)=2+S^*\left(\sum_{n=2}^\infty F_{n-1}+F_{n-2}\right)=2+(S^*({F})-1)+S^*({F})$ and thus $S^*({F})=-1$. As second example, we take the geometric series $s_z=1+z+z^2+z^3+\cdots$. We have $S^*(s_z)=\frac{1}{1-z}$, $z\neq1$, since $S^*(s_z)=1+zS^*(s_z)$. In particular, for $z=-1$ we recover the usual value $S^*(\sigma_0)=S^*(1-1+1-1+\cdots)=\frac{1}{2}$. The same conclusion is true for any positive integer $k$. This is remarked by Knopp \cite[p.\ 479]{Knopp}, but it is not proved there. We include a simple proof using induction on $k$.

\begin{proposition}\label{Main example} Let $S^*$ be a summability method satisfying rules 2 and 3. If $S^\ast$ sums the series $\sigma_k$ for all integers $k\geq0$, then $S^*(\sigma_k)=\frac{2^{k+1}-1}{k+1}B_{k+1}$, for all $k\geq1$.
\end{proposition}

\begin{proof}For $k=1$ we note that  
	$$(1-2+3-4+5-6+\cdots)+(0+1-2+3-4+5-\cdots)=1-1+1-1+1-1+\cdots.$$  Then by rules 2 and 3, $S^*(\sigma_0)=S^*(\sigma_1)+S^*(\sigma_1)=2S^*(\sigma_1)$ and $S^*(\sigma_1)=\frac{1}{2}S^*(\sigma_0)=\frac{1}{4}$. Now we assume the formula holds for $S^*(\sigma_1),...,S^*(\sigma_{k-1})$, $k\geq 2$. By the binomial theorem, we see that 
	$$S^*(\sigma_k)=S^*\left(1-\sum_{n=2}^{\infty}(-1)^{n}(1+(n-1))^k\right)=\frac{1}{2}-\sum_{j=1}^k {k\choose j} S^*(\sigma_j).$$ Using the induction hypothesis and equation (\ref{Prop B_n for 1^k-2^k+...}), we conclude that the formula is valid for $k$. The principle of induction allows us to conclude the proof.
\end{proof}

The previous reasoning provides another way to prove recursion (\ref{Prop B_n for 1^k-2^k+...}): Take $S^*=A$ as Abel summability. Since $A$ satisfies rules 2 and 3, we can replace the value $A(\sigma_k)=\frac{2^{k+1}-1}{k+1}B_{k+1}$ in the recursion obtained in the previous proof. In fact, we can easily generalize equation (\ref{Prop B_n for 1^k-2^k+...}) using the same type of argument.

\begin{proposition}\label{Main result}Let $a$ and $k$ be positive integers. Then the Bernoulli numbers satisfy the recursion formula
	$$\frac{2^{k+1}-1}{k+1}B_{k+1}=\sum_{n=0}^{a-1} (-1)^n(n+1)^k+\frac{(-1)^a}{2}a^k+(-1)^a\sum_{j=1}^{k}{k\choose  j} a^{k-j}\frac{2^{j+1}-1}{j+1}B_{j+1}.$$
\end{proposition}

\begin{proof}The divergent series $a^{-k}\sigma_k=\sum_{n=0}^\infty (-1)^n\left(\frac{n+1}{a}\right)^k$ is Abel summable and $A(a^{-k}\sigma_k)=\frac{1}{a^k}\frac{2^{k+1}-1}{k+1}B_{k+1}$. The formula follows from the binomial theorem and rules 2 and 3 since \begin{align*}
	A\left(a^{-k}\sigma_k\right)&=\frac{1}{a^k}\sum_{n=0}^{a-1}(-1)^n(n+1)^k+A\left(\sum_{n=a}^\infty (-1)^n\left(\frac{n-a+1}{a}+1\right)^k\right)\\
	&=\frac{1}{a^k}\sum_{n=0}^{a-1}(-1)^n(n+1)^k+(-1)^a A\left(\sum_{n=0}^\infty (-1)^n\sum_{j=0}^k {k\choose j}\frac{(n+1)^j}{a^j}\right).
	\end{align*}
\end{proof}

Formula (\ref{Prop B_n for 1^k-2^k+...}) corresponds to the case $a=1$ of Proposition \ref{Main result}. The formula above is simple in the sense that it can be deduced directly from (\ref{Last formula}) by equating corresponding coefficients of $z^k$ in the identity $\sum_{n=0}^{a-1} (-1)^ne^{(n+1)z}+(-1)^ae^{az}-(-1)^a\frac{e^{az}}{e^z+1}=1-\frac{1}{1+e^z}$.

We can go further and recover the usual formulas to determine the Bernoulli numbers in terms of Euler numbers and vice versa.

\begin{proposition}\label{Bernoulli and Euler} The Bernoulli and Euler numbers are related by the formulas \begin{equation}\label{First relation Bn En}\sum_{j=1}^k {k\choose j} 2^j \frac{2^{j+1}-1}{j+1}B_{j+1}=\frac{1}{2}-\frac{(-1)^{\lfloor k/2\rfloor} E_{k}}{2},\end{equation}
	
	\begin{equation}\label{Bernoulli in terms of Eulers}2^{k+1}\frac{2^{k+1}-1}{k+1}B_{k+1}=\sum_{j=0}^{k}{k\choose j}(-1)^{\lfloor j/2\rfloor}E_j,\end{equation} valid for all integers $k\geq1$.
\end{proposition}	

\begin{proof}To prove (\ref{First relation Bn En}), we first calculate the Abel sum of $1^k-3^k+5^k-7^k+\cdots$ using equation (\ref{Abel sum of 1k-3k+5k+...}). Then we use the binomial theorem and rules 2 and 3 to obtain $$A\left(\sum_{n=0}^\infty (-1)^{n}(2n+1)^k\right)=\frac{1}{2}-\sum_{j=1}^k {k\choose j} 2^j \frac{2^{j+1}-1}{j+1}B_{j+1}.$$ Equating both results we get (\ref{First relation Bn En}). Similarly, for (\ref{Bernoulli in terms of Eulers}) it is enough to consider the series $2^k\sigma_k=2^k-4^k+6^k-8^k+\cdots$ and the relation $$A(2^k\sigma_k)=A\left(\sum_{n=1}^\infty (-1)^{n+1}(2n-1+1)^k\right)=\sum_{j=0}^{k}{k\choose j} A\left(\sum_{n=1}^\infty (-1)^{n+1}(2n-1)^j\right).$$ 
\end{proof}

The formulas (\ref{First relation Bn En}) and (\ref{Bernoulli in terms of Eulers}) are of course elementary. They can be deduced by equating the coefficients of $z^k$ in $e^z\frac{-1}{1+e^{2z}}=\frac{-1}{e^z+e^{-z}}$, $e^z\frac{1}{e^{z}+e^{-z}}=\frac{1}{1+e^{-2z}}$, respectively.

We invite the reader to calculate the Abel sum of $\sum_{n=0}^\infty (-1)^{n}(an+q)^k$, $k\geq1$, $a,q\in \mathbb{R}$ and $a>0$ as we did here (setting $x=e^{-ay}$ in the generating series and using rules 2 and 3) to conclude that Bernoulli and Euler numbers also satisfy $$\frac{q^k}{2}-\sum_{j=1}^k {k\choose j}  q^{k-j}a^j \frac{2^{j+1}-1}{j+1}B_{j+1}=\frac{(-1)^k}{2}\sum_{j=0}^{k} {k\choose  j} \left(\frac{a}{2}-q\right)^{k-j}\left(\frac{a}{2}\right)^{j}(-1)^{\lfloor j/2\rfloor}E_{j}.$$ Unfortunately, this expression does not provide new information since it can be deduced directly from (\ref{Bernoulli in terms of Eulers}) by equating the corresponding coefficient of $q^{k-j}a^j$.

\

\begin{remark*}Not all series can be summed with methods satisfying rules 1 to 3. For instance, the sum $s$ of $1+1+1+\cdots$ must satisfy $s=1+s$ which is impossible for a finite value. Another example is the series $\sigma=\sum_{n=1}^\infty n$: if it would be summable for some $S^*$, then $$S^*(\sigma)-S^*(\sigma)=S^*(1+2+3+\cdots)-S^*(0+1+2+\cdots)=S^*(1+1+\cdots),$$ and $1+1+1+\cdots$ would be summable with sum equals to $0$. However, there are methods that assign the controversial value $-\frac{1}{12}$ to $\sigma$. For instance, interpreting $\sigma$ as the value of the analytic continuation of $\zeta(z)$ at $z=-1$. Another example is the \textit{constant of a series method} of Ramanujan \cite[p.\ 327, 346]{Ramanujan,Candel,Hardy}. Naturally, such methods can not satisfy the rules 1 to 3. It is curious that Ramanujan wrote \cite[p.\ 135]{Ramanujan} $$\begin{array}{rllllllllll}
	\sigma&=&1&+&2&+&3&+&4&+&\cdots,\\
	4\sigma&=& & &4&+ & & &8&+&\cdots.
	\end{array}$$ Subtracting both equations he found $-3\sigma=1-2+3-4+5-6+\cdots=\frac{1}{4},$ so again $\sigma=-\frac{1}{12}$, although this reasoning is not compatible with our approach.
\end{remark*}

\begin{remark*} All formulas we have obtained here are well-known, elementary and they admit direct proofs by using power series. Thus it is natural to wonder whether the method we used is widely applicable to more complicated recurrences or to general sequences of numbers. This might not be the case since we have used only linear recursions and the binomial theorem. However, this point of view gives a natural interpretation of formulas (\ref{First relation Bn En}), (\ref{Bernoulli in terms of Eulers}) and the one in Proposition \ref{Main result} in terms of the divergent series involved.
\end{remark*}

\end{document}